\documentclass{amsart}

\newtheorem{theorem}{Theorem}[section]
\newtheorem{lemma}[theorem]{Lemma}

\newtheorem{prop}[theorem]{Proposition}
\theoremstyle{definition}

\def\eproof{$\Box$ \medskip}

\newcommand\rs{\hat{\mathbb{C}}}

\newcommand\R{\mathbb{R}}
\newcommand\Ht{\mathbb{H}^3}
\newcommand\Hp{\mathbb{H}^2}

\title{Renormalized volume and the volume of the convex core}
\author{Martin Bridgeman}
\address{Boston College}
\author{Richard D. Canary}
\address{University of Michigan}
\date{\today}
\thanks{Bridgeman was partially supported by NSF grant DMS-1500545. Canary was partially supported by NSF grant, DMS -1306992.}

\begin{document}
\begin{abstract} 
We obtain upper and lower bounds on the difference between  the
renormalized volume  and the volume of the convex core of a convex cocompact
hyperbolic 3-manifold which depend on the injectivity radius of the boundary of
the universal cover of the convex core and the Euler characteristic of the boundary
of the convex core.  
These results generalize results of Schlenker obtained in
the setting of quasifuchsian hyperbolic 3-manifolds.
\end{abstract}

\maketitle

\section{Introduction}
Krasnov and Schlenker \cite{KS08,KS12} studied the
renormalized volume of a convex cocompact hyperbolic 3-manifold. 
Renormalized volume was introduced in the  more general setting of infinite volume conformally compact Einstein manifolds
as a way to assign a finite normalized volume in a natural way (see Graham-Witten \cite{GW}). 
Krasnov and Schlenker's  renormalized volume generalizes earlier work of Krasnov \cite{Khol}
and Takhtajan-Teo \cite{TT} for special classes of hyperbolic 3-manifolds.
In particular, it is closely related to the Liouville action functional studied by Tahktajan-Teo \cite{TT}
and the renormalized volume gives rise to a K\"ahler potential for the Weil-Petersson metric
(see Krasnov-Schlenker \cite[Section 8]{KS08}).

Schlenker  \cite{schlenker-qfvolume} showed that there exists $K>0$ such that
if $M$ is a quasifuchsian hyperbolic 3-manifold, then
$$V_C(M) -K|\chi(\partial M)| \leq V_R(M) \leq V_C(M)$$
where $V_R(M)$ is the renormalized volume of $M$ and $V_C(M)$ is the volume of the
convex core $C(M)$ of $M$.
This inequality, along with a variational  formula for the renormalized volume,
was used by Kojima-McShane \cite{KMc} and Brock-Bromberg \cite{brock-bromberg}
to give an upper bound on the volume of a hyperbolic
3-manifold fibering over the circle in terms of the entropy of its monodromy map.

In this paper, we use the  work of the authors 
\cite{bridgeman,BC03,BC05,BC10,cbbc} to generalize Schlenker's result to the
setting of all convex cocompact hyperbolic \hbox{3-manifolds}. We exhibit bounds on the difference between $V_C(M)$
and $V_R(M)$  in terms of the injectivity radius of the boundary
of the universal cover of the convex core and the Euler characteristic of the boundary of the convex core.
We will see that, even if $|\chi(\partial C(M))|$ is bounded, this difference can be arbitrarily large.

The {\em convex core} $C(M)$ of a complete hyperbolic 3-manifold $M$ (with non-abelian fundamental group) 
is the smallest convex submanifold of $M$
whose inclusion into $M$ is a homotopy equivalence. Its boundary $\partial C(M)$ is a hyperbolic
surface in its intrinsic metric (see Epstein-Marden \cite[Theorem II.1.12.1]{EM87} and 
Thurston \cite[Proposition 8.5.1]{ThBook}).
A complete hyperbolic \hbox{3-manifold} $M$  (with non-abelian fundamental group) 
is said to be {\em convex cocompact} if $C(M)$ is compact.

Our results, and their proofs, naturally divide into two cases, depending on
whether the boundary of the convex core is incompressible. We recall
that $\partial C(M)$  is {\em incompressible} if whenever $S$ is a component of
$\partial C(M)$, then $\pi_1(S)$ injects into $\pi_1(M)$. Equivalently, the boundary of the convex core is
incompressible if and only if $\pi_1(M)$ is freely  indecomposable. In particular, if $M$ is a quasifuchsian
hyperbolic 3-manifold, the boundary of its convex core is incompressible.
In this case, we get the following generalization of Schlenker's result. 

\begin{theorem}
\label{main incomp}
If $M=\mathbb H^3/\Gamma$ is a convex cocompact hyperbolic \hbox{3-manifold} and
$\partial C(M)$ is incompressible, then
$$V_C(M) - 6.89 |\chi(\partial C(M))| \leq V_R(M) \le V_C(M).$$
Moreover, $V_R(M)=V_C(M)$ if and only if $\partial C(M)$ is totally geodesic.
\end{theorem}

In Proposition \ref{chi necessary} we construct examples demonstrating the necessity of a
linear dependence on $|\chi(\partial C(M))|$ in Theorem \ref{main incomp}.

If the boundary of the convex core is compressible, then the boundary of the universal
cover $\widetilde{C(M)}$ of the convex core is not simply connected and it is natural to consider its
injectivity radius $\eta$, in its intrinsic metric. Equivalently, $\eta$ is half the length of the shortest
homotopically non-trivial curve in  $\partial C(M)$ which bounds a disk in $C(M)$.

\begin{theorem}
\label{main2}
If $M$ is a convex cocompact hyperbolic 3-manifold,
$\partial C(M)$ is compressible and $\eta > 0$ is the injectivity radius of the intrinsic metric 
on $\partial \widetilde{C(M)}$, then
$$V_C(M) -|\chi(\partial M)|  \left(45\log\left(\frac{1}{\min\{1,\eta\} }\right)+67\right)\leq V_R(M) < V_C(M)$$
Furthermore, if $\eta\le \sinh^{-1}(1)$, then
$$V_R(M) \leq V_C(M) -\pi \log\left(\frac{1}{\eta }\right)-1.79.$$
\end{theorem}

If $M=\Ht/\Gamma$, then the {\em domain of discontinuity} $\Omega(\Gamma)$ is the largest open 
subset of $\rs=\partial\Ht$ which $\Gamma$
acts properly discontinuously on. The quotient $\partial_cM=\Omega(\Gamma)/\Gamma$ is called the
{\em conformal boundary} of $M$. The manifold $M$ is convex cocompact if and only if 
$$\widehat M=M\cup\partial_cM=(\Ht\cup\Omega(\Gamma))/\Gamma$$
is compact. $\Omega(\Gamma)$ admits a unique conformal metric of curvature $-1$, called the
{\em Poincar\'e metric}. Since the Poincar\'e metric is conformally natural, it descends to a hyperbolic
metric on the conformal boundary. We also obtain a version of our theorem where the bounds
depend on the injectivity radius of the Poincar\'e metric on $\Omega(\Gamma)$.

\begin{theorem}
\label{main}
If $M=\mathbb H^3/\Gamma$ is a convex cocompact hyperbolic 3-manifold, 
$\partial C(M)$ is compressible and $\nu > 0$ is the injectivity radius of the 
Poincare metric on $\Omega(\Gamma)$, then
$$V_C(M) - |\chi(\partial C(M))|\left(\frac{205}{\nu}+202\right) \leq V_R(M) < V_C(M). $$
Furthermore, if $\nu\le \frac{1}{2}$, then
$$V_R(M)\le V_C(M) - \left( \frac{9}{\nu}-9\right)$$
\end{theorem}

One may loosely reformulate Theorem \ref{main2} as saying that $V_C(M)-V_R(M)$ is comparable to
$\log\frac{1}{\eta(M)}$ when $\eta(M)$ is small,  where $\eta(M)$ is the injectivity radius of $\widetilde{\partial C(M)}$.
Similiarly, one may reformulate Theorem \ref{main} as saying that $V_C(M)-V_R(M)$ is comparable to
$\frac{1}{\nu(M)}$ when $\nu(M)$ is small,  where $\nu(M)$ is the injectivity radius of $\Omega(\Gamma)$ in the Poincar\'e
metric. 

We note that one may obtain slightly more precise forms of  our results by giving exact forms
for the constants involved, but the expressions for the constants would be rather unpleasant and it seems
unlikely that the constants obtained by our techniques are sharp. However, our estimates are of
roughly the correct asymptotic form as $\nu$ or $\eta$ approach 0.

\medskip\noindent
{\bf Acknowledgements:} 
The authors  would also like to thank Curt McMullen and Greg McShane for useful conversations related to this work.
This material is based upon work supported by the National Science Foundation under grant No. 0932078 000 while
the authors were in residence at the Mathematical Sciences Research Institute in Berkeley, CA, during the
Spring 2015 semester.

\section{Renormalized Volume}

In this section, we recall the work of Krasnov-Schlenker (\cite{KS08,KS12}) and 
Schlenker (\cite{schlenker-qfvolume}) on renormalized volume for convex cocompact hyperbolic 3-manifolds. 
We will assume for the remainder of the paper that $M=\Ht/\Gamma$ is convex cocompact.

If $N$ is a compact, $C^{1,1}$  strictly convex submanifold such that  the inclusion of $N$ into $M$
is a homotopy equivalence, the {\em $W$-volume} of $N$ is given by
$$W(N) = V(N) - \frac{1}{2}\int_{\partial N} H dA$$
where $H$ is the mean curvature function.\footnote{We are using the convention that the mean curvature $H$ is the 
average of the principal curvatures, while Krasnov and Schlenker \cite{KS08,KS12} use the convention that $H$ is the
sum of the principal curvatures, so our definition, although apparently different, agrees with theirs.}
(We recall that a submanifold $N$ is strictly convex if the interior of any  geodesic in $M$ joining two
points in $N$ lies in the interior of $N$.)

%

Notice that if $N$ is $C^{1,1}$, then the curvature and mean curvature of $\partial N$ are defined
almost everywhere and the integral of mean curvature is well-defined and well-behaved. This is the
natural regularity assumption, since a metric neighborhood of the convex core is $C^{1,1}$
(see Epstein-Marden \cite[Lemma II.1.3.6]{EM87}), but need not be $C^2$.

If $r>0$ and $N_r$ is the closed $r$-neighborhood of $N$, then $N_r$ is $C^{1,1}$ and strictly convex, and
$\{S_r=\partial N_r\}_{r>0}$ is a family of equidistant surfaces foliating the end of $M$. In particular,
$N_r$ is homeomorphic to $\widehat M$ for all $r>0$.
The following fundamental lemma relates $W(N_r)$ to $W(N)$.

\begin{lemma}
\label{fundamental lemma}
{\rm (Krasnov-Schlenker \cite[Lemma 4.2]{KS08}, Schlenker \cite[Lemma 3.6]{schlenker-qfvolume})}
If $M$ is a convex cocompact hyperbolic 3-manifold and $N$ is a strictly convex, $C^{1,1}$, compact submanifold 
such that  the inclusion of $N$ into $M$ is a homotopy equivalence, then
$$W(N_r) = W(N) -r\pi\chi(\partial C(M)).$$
\end{lemma}

Lemma \ref{fundamental lemma} follows from the fact that 
$$\dot{W}_t = \frac{d}{dt}W(N_t) = \frac{d}{dt} V(N_t) - \frac{1}{2} \frac{d}{dt}\left(\int_{S_t} H_t dA_t\right) = 
A(t) - \frac{1}{4}A''(t) $$
where $A(t)$ is the area of $S_t$. The general solution  to the equation $y''-4y=0$ is $ae^{2t}+be^{-2t}$. Therefore
as the exponential terms in $A(t)$  are of this form, they do not contribute to a change in $W$-volume.  
Further analysis shows that the remaining terms give $\dot{W}_t = -\pi\chi(\partial N)$.

If $I_r$ is the intrinsic metric  on $S_r$,  the normal map identifies $S_r$ with the conformal boundary $\partial_cM$
and one may define the limiting conformal metric $I^*$ on $\partial_c M$ by 
$$I^* = \lim_{r\rightarrow \infty} 4e^{-2r} I_r.$$

C. Epstein  \cite{epstein-envelopes} showed that given any conformal $C^{1,1}$ metric $h$ on $\partial_c M$, 
there exists an (asymptotically) unique family of equidistant submanifolds $N_r(h)$, called  the 
{\em Epstein submanifolds}  whose limiting conformal structure is $h$. 
Explicitly, let $\Omega \subseteq \rs$
be a hyperbolic domain in the Riemann sphere and let  $g$ be a $C^{1,1}$ conformal metric on $\Omega$. 
Given $z\in \Omega$, let $H(z,g)$ be the horoball bounded by the horosphere
$$h(z,g) = \left\{ x\in \Ht\ |\ v_x(z) = g(z)\right\}$$
where $v_x$ is the visual metric on $\rs$ obtained by identifying $\rs$ with $T_x^1\Ht$.
Then
$$\Sigma(g) =\partial \left(\bigcup_{z\in\Omega} H(z,g)\right).$$
is the outer envelope of the collection of horospheres \hbox{$\{h(z,g)\}_{z\in \Omega}$.}

If $h$ is a conformal metric on $\partial_c M$, then $h$ lifts to
a metric $\tilde h$ on $\Omega(\Gamma)$ .
For all sufficiently large $r$,
$\Sigma(e^r \tilde h)$ descends to a $C^{1,1}$ surface $S_r$ bounding a strictly convex
submanifold $N_r(h)$ of $M$. 
Lemma \ref{fundamental lemma} indicates that it is natural
to define the W-volume of $h$ as
$$W(h) = W(N_r(h)) + r \pi\chi(\partial N_r(h))$$
for any $r$ large enough that $N_r(h)$ is well-defined, strictly convex and $C^{1,1}$.

The {\em renormalized volume} $V_R(M)=W(\rho)$ where $\rho$ is the Poincar\'e metric on the
conformal boundary $\partial_c(M)$. Krasnov and Schlenker \cite[Section 7]{KS08} showed 
that the renormalized volume is the maximum of $W(h)$ as $h$ varies over
all smooth conformal metrics on $\partial_cM$ with area $2\pi|\chi(\partial_cM)|$. 

The $W$-volume satisfies the following linearity and monotonicity properties, which will
be very useful in establishing our bounds.

\begin{lemma}
\label{monotonicity}
{\rm (Schlenker \cite[Proposition 3.11,Corollary 3.8]{schlenker-qfvolume}\footnote{The references here and elsewhere in
the paper are to the revised version of \cite{schlenker-qfvolume} which appears at arXiv:1109.6663. In particular, the assumption
that $g$ and $h$ are non-positively curved is omitted from the published version.})}
Let $M$ be a convex cocompact hyperbolic manifold. Then
\begin{enumerate}
\item (Linearity) If $s \in \R$ and $h$ is a $C^{1,1}$ conformal metric on $\partial_cM$,  then 
$$W(e^{s}h) = W(h) -s\pi\chi(\partial M).$$ 
\item (Monotonicity) If $g$ and $h$ are non-positively curved, $C^{1,1}$, conformal metrics on $\partial_c M$ 
and $g(x) \leq h(x)$ for all $x \in \partial_c M$, then 
$$W(g) \leq W(h).$$
\end{enumerate}
\end{lemma}

The proof of (1) follows nearly immediately from the definitions. We note that from the definition of $N_r(h)$ that 
$N_r(e^{s}h) = N_{r+s}(h)$. Therefore,
\begin{eqnarray*}
W(e^{s}h) & = & W(N_r(e^{s}h)) + \pi r\chi(M) = W(N_{r+s}(h)) + \pi r\chi(\partial M) \\
 & = & (W(N_r(h))-\pi s\chi(\partial M)) + \pi r\chi(\partial M) \\
 & = & (W(N_r(h)) +\pi r\chi(\partial M)) -\pi s\chi(\partial M)\\
 & =  &  W(h)-\pi s\chi(\partial M).\\
\end{eqnarray*}

The proof of (2) is more involved. One first observes that if $g \leq h$ and $r$ is large enough that
$N_r(g)$ and $N_r(h)$ are both defined, then $N_r(g)\subseteq N_r(h)$. Schlenker then defines
a relative $W$-volume of the region \hbox{$N_r(h)-N_r(g)$}, which agrees with \hbox{$W(N_r(h))-W(N_r(g))$},
and uses a foliation of \hbox{$N_r(h)-N_r(g)$}  by strictly convex, $C^{1,1}$, non-positively curved surfaces to
prove that this relative $W$-volume is non-negative.
 
\section{The Thurston metric on the conformal boundary}
\label{Thurston metric}

The {\em Thurston metric} $\tau=\tau(z)|dz|$ on  a hyperbolic domain $\Omega\subset\rs$ is defined 
by letting the length of a vector
$v \in T_{z}(\Omega)$ be the infimum of the hyperbolic length of all vectors
$v' \in \Hp$ such that there exists a M\"obius transformation 
$f$ such that $f(\Hp)\subset \Omega$ and $df(v') = v$. The Thurston
metric is clearly conformally natural and conformal to the Euclidean metric.
Therefore, if $M=\Ht/\Gamma$ is convex cocompact, then the Thurston metric $\tau$ on $\Omega(\Gamma)$
descends to a conformal metric on $\partial_cM$ which we will again denote $\tau$ and call the Thurston metric.
Kulkarni and Pinkall \cite[Theorem 5.9]{KP} proved that the Thurston metric is $C^{1,1}$ and
non-positively curved (see also Herron-Ibragimov-Minda \cite[Theorem C]{HIM}).

We recall that the Poincar\'e metric  $\rho=\rho(z)|dz|$ on $\Omega$ can be similarly defined by letting
the length of a vector $v \in T_z(\Omega)$ be the infimum of the hyperbolic length over  all vectors
$v'$ such that there exists a conformal map $f:\Hp \rightarrow \rs$ such that
$f(\Hp)\subset \Omega$ and  $df(v') = v$. So, by definition, 
$$\rho(z)\le\tau(z)$$ for all $z\in\Omega$.  
So, by the monotonicity
lemma, Lemma \ref{monotonicity}, 
$$V_R(M)=W(\rho) \leq W(\tau).$$

One may combine estimates of Beardon-Pommerenke  \cite{BP}, Canary \cite[Corollary 3.3]{cbbc} and
Kulkarni-Pinkall \cite[Theorem 7.2]{KP} to establish the following relationship between the Poincar\'e metric
and the Thurston metric  of a uniformly perfect hyperbolic domain (see Bridgeman-Canary \cite[Section 3]{BC10}).
Notice that if $M=\Ht/\Gamma$ is convex cocompact, then $\Gamma$ acts cocompactly by isometries on
$\Omega(\Gamma)$, so there is a lower bound on the injectivity radius of $\Omega(\Gamma)$ in the Poincar\'e
metric.

\begin{theorem}
Let $\Omega$ be a hyperbolic domain in $\rs$ and let $\nu > 0$ be the injectivity radius of the Poincare metric 
$\rho$ on $\Omega$. If $\tau$ is the Thurston metric on $\Omega$ and $k =4+\log(3+ 2\sqrt{2})\approx 5.76$, then
$$\frac{\tau(z)}{2\sqrt{2}(k+\frac{\pi^2}{2\nu})}\le  \rho(z)\le \tau(z)$$
for all $z\in\Omega$.
Moreover, $\rho=\tau$ if and only if $\Omega$ is a round disk.
\label{bilip}
\end{theorem}

If $\Omega$ is a simply connected hyperbolic domain, then 
the Thurston metric and the Poincar\'e metric are 2-bilipschitz.

\begin{theorem}{\rm (Anderson \cite[Thm. 4.2]{anderson-thesis}, Herron-Ma-Minda \cite[Lemma 3.2]{HMM05})}
\label{bilip incomp}
If $\Omega$ is a simply connected hyperbolic domain with Poincare metric $\rho$ and Thurston metric $\tau$, then
$$\frac{\tau(z)}{2}\le\rho(z)\le\tau(z)$$
for all $z\in\Omega$.
\end{theorem}

It will be useful to be able to pass back and forth between  lower bounds on the injectivity radius of the boundary 
$\partial\widetilde{C(M)}$ of
the universal cover of the convex core, in the intrinsic metric, and lower bounds on the injectivity radius bound of
the Poincar\'e metric on the domain of discontinuity. 

\begin{prop} 
\label{inj back and forth}
Suppose that $M=\Ht/\Gamma$ is a convex cocompact hyperbolic 3-manifold and that $\partial C(M)$ is non-empty.
\begin{enumerate}
\item
{\rm (Bridgeman-Canary \cite[Lemma 8.1]{BC03})} If $\nu>0$ is a lower bound for the injectivity radius of
$\Omega(\Gamma)$ in the Poincar\'e metric, then 
$$\frac{e^{-m}e^{\frac{-\pi^2}{2\nu}}}{2}$$
is a lower bound for the injectivity radius of $\partial\widetilde{C(M)}$ in its intrinsic metric, where $m = \cosh^{-1}(e^2) \simeq 2.68854$.
\item
{\rm (Canary \cite[Theorem 5.1]{cbbc})} If $\eta>0$ is a lower bound for the injectivity radius of $\partial\widetilde{C(M)}$ 
in its intrinsic metric, then 
$$\min\left\{ \frac{1}{2}, \frac{\eta}{.153}\right\}$$
is a lower bound for the injectivity radius of $\Omega(\Gamma)$ in the Poincar\'e metric.
\end{enumerate}
\end{prop}

\medskip\noindent
{\bf Remark:}
The Thurston metric is also known as the projective (or grafting)
metric, as it arises from regarding $\Omega$ as a complex projective surface
and giving it the metric Thurston described on such surfaces (see Tanigawa \cite[Section 2]{tanigawa} or
McMullen \cite[Section 3]{mcmullenCE} for further details).
Kulkarni and Pinkall  \cite{KP} defined and studied a generalization
of this metric in all dimensions and it is also sometimes called the 
Kulkarni-Pinkall metric.

\section{The bending lamination and renormalized volume}

The boundary of the convex core of a convex cocompact hyperbolic \hbox{3-manifold} $M=\Ht/\Gamma$
is a hyperbolic surface in its intrinsic metric. It is totally geodesic
except along a lamination $\beta_M$, called the {\em bending lamination}. The bending lamination inherits a transverse
measure which records the degree to which the surface is bent along the lamination. The length
$L(\beta_M)$ of the bending lamination then records the total amount of bending of the convex core 
(see Epstein-Marden \cite[Section II.1.11]{EM87} for details on the bending lamination).

If $N_r$ is the closed $r$-neighborhood of $C(M)$ for all $r>0$, then one can easily check that
$\{\tilde S_r=\partial \tilde N_r\}_{r>0}$ is a family of Epstein surfaces for the Thurston metric on
$\Omega(\Gamma)$ (see Bridgeman-Canary \cite[Lemma 3.5]{BC10} for example).  
Using this observation, one may establish the following equality:

\begin{lemma}
\label{bending and Wvolume}
{\rm (Schlenker \cite[Lemma 4.1]{schlenker-qfvolume})}
If $M$ is a convex cocompact hyperbolic 3-manifold, $\partial C(M)$ is non-empty and
$\beta_M$ is the bending lamination, then
$$W(\tau) = V_C(M) - \frac{1}{4}L(\beta_M).$$
where $\tau$ is the Thurston metric on $\partial_cM$.
\end{lemma}

Furthermore, we have the following bounds on the length of the bending lamination of the convex core
in terms of the injectivity radius of the Poincar\'e metric on the domain of discontinuity.

\begin{theorem}{\rm (Bridgeman-Canary \cite[Theorem $1'$, Theorem $2'$]{BC05})}
\label{bending length}
If \hbox{$M=\Ht/\Gamma$} is a convex cocompact hyperbolic 3-manifold 
and $\nu > 0$ is the  injectivity radius of the Poincare metric on $\Omega(\Gamma)$, then
$$ L(\beta_M) \leq |\chi(\partial M)| \left(\frac{807}{\nu}+771\right).$$
Furthermore, if $\nu \leq 1/2$, then 
$$L(\beta_M ) \geq \frac{37}{\nu}-36.$$
\end{theorem}

We also have a bounds  on $L(\beta_M)$ in terms of the injectivity radius of $\partial\widetilde{C(M)}$
in its intrinsic metric.

\begin{theorem}{\rm (Bridgeman-Canary \cite[Theorem 1,  Theorem 2]{BC05})}
\label{bending length CH bound}
If \hbox{$M=\Ht/\Gamma$} is  a convex cocompact hyperbolic 3-manifold 
and $\eta> 0$  is the  injectivity radius of the intrinsic metric on $\partial\widetilde{C(M)}$, then 
$$L(\beta_M) \leq |\chi(\partial M)| \left(164\log\left(\frac{1}{\min\{1,\eta\}}\right)+218\right).$$
Furthermore, if $\eta\le \sinh^{-1}(1)$, then
$$L(\beta_M)\ge 4\pi \log\left(\frac{2\sinh^{-1}(1)}{\eta}\right).$$
\end{theorem}

If $C(M)$ has incompressible boundary, we obtain the following bound which improves on the bound
obtained in Theorem 3 in \cite{BC05}. (A similar argument is given in the proof of Theorem 6.7 in 
Anderson \cite{anderson-thesis}.)

\begin{theorem}\label{bending length incomp}
If $M$ is a convex cocompact hyperbolic 3-manifold, $\partial C(M)$ is 
incompressible, and $\beta_M$ is the bending lamination, then
$$L(\beta_M)\le 6\pi |\chi(\partial C(M))|.$$
\end{theorem}

\begin{proof}
Recall that, by Theorem  \ref{bilip incomp}, $\tau(z) \leq  2\rho(z)$ for all $z\in\Omega(\Gamma)$, so 
$$Area_\tau(\partial M)  = \int_{\partial M} \tau^2 \leq 4 \int_{\partial M} \rho^2 = 4Area_\rho(\partial M) = 4(2\pi|\chi(\partial M)|).$$
A simple calculation shows that $Area_\tau(\partial M) = 2\pi|\chi(\partial M)| + L(\beta_M)$  (see 
Schlenker \cite[Section 4.2]{schlenker-qfvolume}). 
Therefore,
$$2\pi|\chi(\partial M)| + L(\beta_M) \leq 4(2\pi|\chi(\partial M)|),$$
which implies that
$$L(\beta_M) \leq 6\pi|\chi(M)|.$$
\end{proof}

\medskip\noindent
{\bf Remark:} One may use the proof of Theorem \ref{bending length incomp} and
the estimate from Theorem \ref{bilip} to bound
the length of the bending locus in the compressible case. However, in this situation the
argument gives that
$$L(\beta_M)\le \left(16\pi\left(k+\frac{\pi^2}{2\nu}\right)^2-2\pi\right)|\chi(\partial M)|$$
which is significantly worse than the bound obtained in Theorem \ref{bending length}.

\section{Proofs of main results}

We have now assembled the necessary ingredients to prove our main results.
We begin by proving Theorem \ref{main incomp} which gives the bounds in the
simplest case where the convex core has incompressible boundary.

\medskip

{\bf Proof of Theorem \ref{main incomp}:}
Suppose that $M$ is a convex cocompact hyperbolic 3-manifold such that $\partial C(M)$ is
incompressible. Let $\rho$ be the Poincar\'e metric on $\partial_cM$ and
let $\tau$ be the Thurston metric on $\partial_cM$.
Theorem \ref{bilip incomp} implies that 
$$\frac{\tau}{2}\le\rho\le\tau,$$
so the monotonicity lemma, Lemma \ref{monotonicity},
implies that
$$W(\tau)+\pi\log(2)\chi(\partial M)=W\left(\frac{\tau}{ 2}\right)\le W(\rho) \le W(\tau).$$
Theorem \ref{bending length incomp} implies that 
$$L(\beta_M)\le 6\pi |\chi(\partial C(M))|$$
and Lemma \ref{bending and Wvolume} implies that
$$W(\tau) = V_C(M) - \frac{1}{4}L(\beta_M)\le V_C(M).$$
It follows that 
$$V_C(M)-\left(\pi\log(2)+\frac{6\pi}{ 4}\right)|\chi(\partial C(M))|\le W(\rho)\le V_C(M).$$
Since $V_R(M)=W(\rho)$ and $\pi\log(2)+\frac{6\pi}{ 4 }\le 6.89 $, it follows that
$$V_C(M)-6.89|\chi(\partial C(M))|\le V_R(M)\le V_C(M)$$
as claimed.

If $\partial C(M)$ is totally geodesic, then every component of $\Omega(\Gamma)$ is a round
disk, so $\rho=\tau$, $L_\beta(M)=0$ and $W(\tau)=V_C(M)=V_R(M)=W(\rho)$. On the other hand, if
$V_C(M)=V_R(M)$, then $L(\beta_M)=0$, so $\partial C(M)$ is totally geodesic. Therefore, $V_R(M)=V_C(M)$
if and only if $\partial C(M)$ is totally geodesic.
\eproof

\begin{prop}
\label{chi necessary}
There exists a sequence $\{M_n\}$ of quasifuchsian hyperbolic 3-manifolds such that
$$\lim_{n \rightarrow \infty} V_C(M_n)-V_R(M_n)=+\infty$$ 
and there exists $D>0$ such that 
$$ \frac{V_C(M_n)-V_R(M_n)}{|\chi(\partial M_n)|} \ge D$$
for all $n$.
\end{prop}

\begin{proof}
Let $M$ be a quasifuchsian hyperbolic 3-manifold such that \hbox{$L(\beta_M)\ne0$}.
Let $\{\pi_n:M_n\to M\}$ be a sequence of finite covers of $M$ whose degrees $\{d_n\}$ tend to
infinity.  The convex core $C(M_n)=\pi_n^{-1}(C(M))$ and similarly
the bending lamination  $\beta_{M_n}$ is the pre-image of $\beta_M$. It follows that
\hbox{$|\chi(\partial C(M_n)|=d_n|\chi(\partial C(M))|$} and
\hbox{$L(\beta_{M_n})=d_nL(\beta_M)$} for all $n$. Since, as we saw in the above proof,
$$V_R(M_n)=W(\rho_n)\le W(\tau_n)=V_C(M_n)-\frac{1}{4}L(\beta_{M_n})$$
where $\rho_n$ is the Poincar\'e metric on $\partial_cM_n$ and $\tau_n$ is the Thurston metric
on $\partial_cM_n$, it follows that
$$V_C(M_n)-V_R(M_n)\ge \frac{1}{4}L(\beta_{M_n})=\frac{d_n}{4}L(\beta_M)=\frac{L(\beta_M)}{4|\chi(\partial C(M))|}|\chi(\partial C(M_n)|.$$
The result follows if we choose $D=\frac{L(\beta(M)}{ 4|\chi(\partial C(M)|}$.
\end{proof}

\medskip

We now prove Theorem \ref{main} which bounds $V_R(M)$ in
terms of $\chi(\partial C(M))$ and  the injectivity radius of the domain of
discontinuity in its Poincar\'e metric.

\medskip

{\bf Proof of Theorem \ref{main}:}
Suppose that $M=\Ht/\Gamma$ is a convex cocompact hyperbolic 3-manifold such that $\partial C(M)$ is
compressible.  Let $\nu>0$ be the injectivity radius of $\Omega(\Gamma)$
in its Poincar\'e metric.
Let $\rho$ be the Poincar\'e metric on $\partial_cM$ and
let $\tau$ be the Thurston metric on $\partial_cM$.

Theorem \ref{bilip} implies that 
$$\frac{\tau}{2 \sqrt{2}(k+\frac{\pi^2}{2\nu})} \le  \rho\le \tau$$
so the monotonicity lemma, Lemma \ref{monotonicity},
implies that
$$W(\tau)+\pi\log\left(2\sqrt{2}\left(k+\frac{\pi^2}{2\nu}\right)\right)\chi(\partial M)=W\left(\frac{\tau}{2 \sqrt{2}(k+\frac{\pi^2}{2\nu})}\right)\le W(\rho)\le W(\tau).$$
Lemma \ref{bending and Wvolume} implies that
$$W(\tau) = V_C(M) - \frac{1}{4}L(\beta_M)<V_C(M)$$
while Theorem \ref{bending length} implies that
$$ L(\beta_M) \leq |\chi(\partial M)| \left(\frac{807}{\nu}+771\right)$$
and, if  $\nu\le 1/2$, then
$$L(\beta_M ) \geq \frac{37}{\nu}-36.$$
Since $W(\rho) = V_R(M)$, we may combine the above estimates to see that  
$$V_C(M) - K_1(\nu)|\chi(\partial M)| \leq V_R(M)<V_C(M)$$
where
$$K_1(\nu)  =  \pi\log\left({2 \sqrt{2}\left(k+\frac{\pi^2}{ 2\nu}\right)}\right)+\frac{1}{4} \left(\frac{807}{\nu}+771\right)$$
As $\log(a+b)\le \log(a)+b/a$ if $a>1$ and $b>0$ we have
 \begin{eqnarray*}
K_1(\nu)  & \le & \pi\left(\log(2k\sqrt{2})+\frac{\pi^2}{2k \nu}\right) + \frac{202}{\nu}+193\\
 & \le & \left(202+\frac{\pi^3}{2k}\right)\left(\frac{1}{\nu}\right)+\left(\pi\log(2k\sqrt{2})+193\right) \\
 & \le & \frac{205}{\nu}+202.\\
 \end{eqnarray*} 
Moreover, if $\nu\le 1/2$, then
$$V_R(M) \leq V_C(M) -\frac{1}{4}\left(\frac{37}{\nu}-36\right)\le V_C(M)-\left(\frac{9}{\nu}-9\right).$$
\eproof

\medskip\noindent
{\bf Remark:} One may apply the technique of proof of Proposition \ref{chi necessary} to produce a
sequence $\{ M_n=\Ht/\Gamma_n\}$  of Schottky hyperbolic 3-manifolds such that the injectivity radius $\nu(M_n)$ of
$\Omega(\Gamma_n)$ is constant, yet 
$$V_C(M_n)-V_R(M_n)\to\infty\ \ \ {\rm and}\ \ \  \liminf \frac{V_C(M_n)-V_R(M_n)}{|\chi(\partial C(M_n))|} >0.$$
Such a sequence demonstrates the
dependence on $|\chi(\partial C(M))|$ is necessary in Theorem \ref{main}. We recall that a convex cocompact
hyperbolic 3-manifold $M$ is called Schottky if $\pi_1(M)$ is a free group.

\medskip

One may derive a version of Theorem \ref{main2} directly from Theorem \ref{main} and Proposition
\ref{inj back and forth}. However, we will obtain better estimates by giving a more direct proof. 

\medskip

{\bf Proof of Theorem \ref{main2}:}
Suppose that $M=\Ht/\Gamma$ is a convex cocompact hyperbolic 3-manifold such that $\partial C(M)$ is
compressible.  Let $\eta>0$ be the injectivity radius of $\partial\widetilde{C(M)}$
in its intrinsic metric.
Let $\rho$ be the Poincar\'e metric on $\partial_cM$ and
let $\tau$ be the Thurston metric on $\partial_cM$.
We will consider the two bounds separately.
As before we have
$$V_R(M) \leq W(\tau)= V_C(M) - \frac{1}{4}L(\beta_M) <V_C(M).$$
If $\eta<\sinh^{-1}(1)$, then Theorem \ref{bending length CH bound}
implies that
\begin{eqnarray*}
V_R(M) & \leq  &V_C(M) -  \pi \log\left(\frac{2\sinh^{-1}(1)}{\eta}\right)\\
 & = & V_C(M) -  \pi \log\left(2\sinh^{-1}(1)\right)-\pi\log\left(\frac{1}{\eta}\right)\\
 & \le & V_C(M)-1.79-\pi\log\left(\frac{1}{\eta}\right)\\
\end{eqnarray*}

Proposition \ref{inj back and forth} implies that
$\min\{1/2,\eta/.153\}$ is a lower bound for the injectivity radius of $\Omega(\Gamma)$ in the Poincar\'e metric.
Theorem \ref{bilip} then implies that
$$\frac{\tau}{2 \sqrt{2}\left(k+\frac{\pi^2}{2\min\{1/2,\eta/.153\})}\right)}\le  \rho\le \tau,$$
so Lemma \ref{monotonicity} implies that
\begin{eqnarray*}
V_R(M) =  V(\rho) &\ge  & W\left(\frac{\tau}{ 2 \sqrt{2}\left(k+\frac{\pi^2}{\min\{1,\eta/.076\}}\right)}\right)\\
   &=  & W(\tau) -\pi\log\left({2 \sqrt{2}\left(k+\frac{\pi^2}{\min\{1,\eta/.076\}}\right)}\right)|\chi(\partial M)|.\\
\end{eqnarray*}
Theorem \ref{bending length CH bound} gives  that
$$L(\beta_M) \leq |\chi(\partial M)| \left(164\log\left(\frac{1}{\min\{1,\eta\}}\right)+218\right),$$
so
$$V_R(M)\ge V_C(M) - K_1'(\eta)|\chi(\partial M)|$$
where
\begin{eqnarray*}
K_1'(\eta) & =  & \pi \log\left(2 \sqrt{2}\left(k+\frac{\pi^2}{\min\{1,\eta/.076\} }\right)\right)+
\frac{1}{4} \left(164\log\left(\frac{1}{\min\{1,\eta\} }\right)+218\right)\\
& \le & \pi \log\left(2 \sqrt{2}\left(k+ \frac{\pi^2}{\min\{1,\eta\}}\right)\right)+
\frac{1}{4} \left(164\log\left(\frac{1}{\min\{1,\eta\} }\right)+218\right)\\
& \le &  \pi \log\left(\frac{1 }{\min\{1,\eta\}}\right) + \pi \log\left(2 \sqrt{2}\left(k\min\{1,\eta\} + \pi^2 \right)\right) +\\
 &    & \frac{1}{4} \left(164\log\left(\frac{1}{\min\{1,\eta\} }\right)+218\right)\\
 & \le & \pi\left(\log\left(2\sqrt{2}(k+\pi^2)\right)\right)+\frac{218}{ 4}+\left(\pi+\frac{164}{4}\right)\log\left(\frac{1}{\min\{1,\eta\} }\right)\\
 &\le &45\log\left(\frac{1}{\min\{1,\eta\} }\right)+67.\\
\end{eqnarray*}
\eproof

\end{document}